\documentclass[12pt,reqno]{amsart}
\usepackage[a4paper, hmargin={2.7cm,2.7cm},vmargin={3.3cm,3.3cm}]{geometry}
\usepackage[english]{babel}
\usepackage{color}
\usepackage[noadjust]{cite}
\usepackage{amsmath}
\usepackage{amssymb}
\usepackage{amsfonts}
\usepackage{amsthm}
\usepackage{enumerate}
\usepackage{amsthm}
\usepackage{dsfont}

\usepackage[textwidth=20mm]{todonotes}
\usepackage[hypertexnames=false]{hyperref}
\usepackage[nameinlink]{cleveref}
\usepackage{commath}

\usepackage{etoolbox}

\usepackage{amssymb}   
\usepackage{amsthm}    
\usepackage{thmtools}  
\usepackage{mathtools} 
\usepackage{mathrsfs}  
\usepackage{commath}
\usepackage{bbm}
\usepackage[textwidth=20mm]{todonotes}

\theoremstyle{plain}
\newtheorem{theorem}{Theorem}[section]
\newtheorem{lemma}[theorem]{Lemma}
\newtheorem{proposition}[theorem]{Proposition}
\newtheorem{corollary}[theorem]{Corollary}

\theoremstyle{definition}

\theoremstyle{remark}

\usepackage[all]{xy}

\newcommand{\N}{\mathbb{N}}

\newcommand{\R}{\mathbb{R}}

\newcommand\numberthis{\addtocounter{equation}{1}\tag{\theequation}}

\DeclareMathOperator{\Sep}{sep}
\DeclareMathOperator{\Rel}{rel}
\DeclareMathOperator{\covol}{covol}

\DeclareMathOperator{\Span}{\overline{span}}

\title[Exponential frames and Riesz sequences  at the critical density]{Exponential frames and Riesz sequences \\ at the critical density}

\author{Ulrik Enstad}
\address{Department of Mathematics,
University of Oslo,
Moltke Moes vei 35,
0851 Oslo.}
\email{ubenstad@math.uio.no}

\author{Jordy Timo van Velthoven}
\address{Faculty of Mathematics,
University of Vienna, 
Oskar-Morgenstern-Platz 1,
1090 Vienna, Austria}
\email{jordy-timo.van-velthoven@univie.ac.at}



\begin{document}

\begin{abstract}
    We characterize exponential systems on sets of finite measure that form a frame or a Riesz sequence at the critical density. Namely, they are precisely those systems for which the underlying point set admits a weak limit that yields a Riesz basis.
    In combination with a recent result by Kozma, Nitzan and Olevskii, this shows that there exist sets that fail to possess a frame or Riesz sequence at the critical density, solving an open problem posed by Olevskii. As another consequence, we show that exponential frames and Riesz sequences over repetitive point sets (such as cut-and-project sets) at the critical density are already Riesz bases.
\end{abstract}

\maketitle

\section{Introduction}
For a set $S \subseteq \mathbb{R}^d$ of finite measure and a discrete set $\Lambda \subseteq \R^d$, we consider the system of exponential functions
\[
E(\Lambda) = \{e_{\lambda} : \lambda \in \Lambda \}, \quad e_{\lambda} (t) = e^{2 \pi i  \lambda \cdot t }, \quad t \in S .
\]
In this paper, we study the frame and Riesz sequence properties of the family $E(\Lambda)$ in the Hilbert space $L^2 (S)$. 
Recall that $E(\Lambda)$ is a Riesz basis for $L^2 (S)$ if it is simultaneously a frame and a Riesz sequence. Clearly, any orthogonal basis $E(\Lambda)$ is a Riesz basis for $L^2 (S)$. However, not every set $S$ admits an orthogonal basis, and the characterization of such sets is related to Fuglede's conjecture, see, e.g., \cite{fuglede1974commuting, lev2022fuglede}. Recently, Kozma, Nitzan and Olevskii \cite{kozma2023set} provided an example of a bounded set $S \subseteq \mathbb{R}$ not admitting any Riesz basis; see, e.g., \cite{kozma2015combining, kozma2016combining, debernardi2022riesz} for various (classes of) sets that do admit Riesz bases. On the other hand, it is known that every set of finite measure admits a frame of exponentials \cite{nitzan2016exponential} and hence, by the Feichtinger conjecture (see, e.g., \cite{bownik2024selector, casazza2006kadison, marcus2015interlacing}), also a Riesz sequence of exponentials. The question of whether any such set admits a frame at the ``critical density'' was posed as an open problem by Olevskii \cite{olevskiipresentation, olevskii2013sampling}. As an application of the main result of this paper, we solve this problem by showing that the set shown not to admit a Riesz basis in \cite{kozma2023set} does not even admit a frame or a Riesz sequence at the critical density.

Before stating our main result, we recall Landau's necessary density conditions \cite{La67, nitzan2012revisiting} for exponential frames and Riesz sequences. If $E(\Lambda)$ is a frame for $L^2 (S)$, then the lower Beurling density of $\Lambda$ satisfies
\[
D^- (\Lambda) :=  \liminf_{r \to \infty} \inf_{x \in \mathbb{R}^d} \frac{\# (\Lambda \cap B_r(x))}{|B_r(0)|} \geq |S|,
\]
where $B_r (x)$ denotes the Euclidean ball and $|U|$ the Lebesgue measure of a set $U \subseteq \mathbb{R}^d$.
On the other hand, if $E(\Lambda)$ is a Riesz sequence in $L^2 (S)$, then the upper Beurling density of $\Lambda$ satisfies 
\[
D^+ (\Lambda) := \limsup_{r \to \infty} \sup_{x \in \mathbb{R}^d} \frac{\# (\Lambda \cap B_r(x))}{|B_r(0)|} \leq |S|.
\]
In particular, any Riesz basis $E(\Lambda)$ for $L^2 (S)$ must have uniform density equal to the critical density $|S|$. When $S$ is an interval in $\mathbb{R}$, these necessary density conditions are also sufficient, except for frames and Riesz sequences at the critical density \cite{Be89, duffin1952class}; see \cite{ortegacerda2002fourier} for a complete characterization of sets yielding exponential frames at an interval.
The sufficient and necessary density conditions for exponential frames and Riesz sequences have served as the starting point for the investigation of sampling and interpolation problems in various other settings in complex and harmonic analysis, see, e.g., \cite{FuGrHa17, RaSt95, enstad2025dynamical, grochenig2018sampling, belov2023gabor, balan2006density, seip2004interpolation, Se93}.

It is not the case that every exponential frame or Riesz sequence at the critical density is a Riesz basis. However, our main theorem shows that this is true in an approximate sense: Namely, we characterize frames (resp.\ Riesz sequences) $E(\Lambda)$ at the critical density as exactly the frames (resp.\ Riesz sequences) for which some weak limit of translates of $\Lambda$ yields a Riesz basis. The notion of weak limits is here to be understood in Beurling's sense \cite{Be89}: A sequence $\{ \Lambda_n \}_{n \in \N}$ of sets $\Lambda_n \subseteq \R^d$ is said to \emph{converge weakly} to a set $\Gamma \subseteq \R^d$ if for every $r > 0$ and $\varepsilon > 0$ there exists $n_0 \in \N$ such that, for all $n \geq n_0$,
\begin{align*}
    \Lambda_n \cap B_r(0) \subseteq \Gamma + B_\varepsilon(0) \quad \text{and} \quad \Gamma \cap B_r(0) \subseteq \Lambda_n + B_\varepsilon(0).
\end{align*}
The set of all weak limits of translates of $\Lambda$ is denoted by $W(\Lambda)$, the so-called \emph{hull} of $\Lambda$.

The following theorem is the main result obtained in this paper.

\begin{theorem} \label{thm:main_intro}
      Let $S \subseteq \mathbb{R}^d$ be a set of finite measure and $\Lambda \subseteq \mathbb{R}^d$ be separated. Then the following assertions hold:
      \begin{enumerate}
          \item[(i)] Suppose $E(\Lambda)$ is a frame for $L^2 (S)$. Then $D^- (\Lambda) = |S|$ if and only if there exists $\Gamma \in W(\Lambda)$ such that $E(\Gamma)$ is a Riesz basis for $L^2 (S)$.
          \item[(ii)] Suppose $E(\Lambda)$ is a Riesz sequence in $L^2(S)$. Then $D^+(\Lambda) = |S|$ if and only if there exists $\Gamma \in W(\Lambda)$ such that $E(\Gamma)$ is a Riesz basis for $L^2(S)$.
      \end{enumerate}
\end{theorem}

\Cref{thm:main_intro} provides a far-reaching generalization of results by Seip \cite{seip1995on} for an interval $S$ in $\mathbb{R}$ to general (possibly unbounded) sets of finite measure. The proofs in \cite{seip1995on} use the sufficient density conditions for sampling and interpolation of certain entire functions due to Beurling \cite{Be89, seip2004interpolation}, which are not available beyond spectra being intervals. In contrast, our proof of \Cref{thm:main_intro} is entirely based on methods from Fourier analysis and frame theory and is essentially self-contained.

Our approach towards proving part (i) of \Cref{thm:main_intro} consists of showing that for a frame $E(\Lambda)$ for $L^2(S)$ with critical density $D^-(\Lambda) = |S|$, there exist points $\{x_n \}_{n \in \mathbb{N}}$ in $\mathbb{R}^d$ such that the exponential systems $E(\Lambda - x_n)$ approach a Riesz basis as $n\to \infty$; see \Cref{thm:frame-to-riesz-basis} (iii) for a precise statement. A key tool in establishing part (i) of \Cref{thm:main_intro} is an identity relating the notion of a frame measure \cite{balan2006density, balan2007measure} and Beurling density, cf. \Cref{thm:framemeasure}. In a similar manner, we prove part (ii) of \Cref{thm:main_intro} by showing that for a Riesz sequence $E(\Lambda)$ in $L^2(S)$ with critical density $D^+(\Lambda) = |S|$ there exist translates $\Lambda - y_n$ for which the exponential systems $E(\Lambda - y_n)$ approach a Riesz basis as $n \to \infty$; see \Cref{thm:riesz-sequence-to-riesz-basis} (iii) for a precise statement. A key tool in establishing part (ii) of \Cref{thm:main_intro} is a seemingly new identity for the Beurling density of a Riesz sequence, cf. \Cref{prop:riesz-sequence-density}. 

We next discuss two applications of \Cref{thm:main_intro}.

\subsection{Sets with no frames or Riesz sequences at the critical density}
Our first application of \Cref{thm:main_intro} is motivated by a recent result of Kozma, Nitzan and Olevskii. In \cite{kozma2023set}, they solved a longstanding problem by constructing a compact set $S \subseteq \mathbb{R}$ of positive measure with the property that $L^2(S)$ admits no Riesz basis of exponentials.
A question that naturally arises is whether every compact set still admits a frame at the critical density. In fact, this was posed as an open problem by Olevskii in \cite{olevskiipresentation, olevskii2013sampling}. A combination of \Cref{thm:main_intro} and \cite[Theorem 2]{kozma2023set} yields, in particular, the following answer to this open problem:

\begin{theorem} \label{thm:noframe_critical}
    The set  $S \subseteq \mathbb{R}$ from \cite{kozma2023set} has the following properties:
    \begin{enumerate}
        \item[(i)] There exists no discrete set $\Lambda \subseteq \mathbb{R}$ with $D^-(\Lambda) = |S|$ such that $E(\Lambda)$ is a frame for $L^2 (S)$.
        \item[(ii)] There exists no discrete set $\Lambda \subseteq \mathbb{R}$ with $D^+(\Lambda) = |S|$ such that $E(\Lambda)$ is a Riesz sequence in $L^2 (S)$.
    \end{enumerate}
\end{theorem}

\begin{proof}[Proof of \Cref{thm:noframe_critical}]
(i) Arguing by contradiction, we assume that there exists a discrete set $\Lambda \subseteq \mathbb{R}$ with density $D^-(\Lambda) = |S|$ such that $E(\Lambda)$ is a frame for $L^2 (S)$. If $\Lambda$ is separated, then \Cref{thm:main_intro} (i) implies the existence of $\Gamma \in W(\Lambda)$ such that $E(\Gamma)$ is a Riesz basis for $L^2(S)$. This contradicts \cite[Theorem 2]{kozma2023set}.

More generally, if $\Lambda$ is merely relatively separated, we may reduce the problem to the separated case as follows: Since $S \subseteq \mathbb{R}$ is bounded, there exists a separated subset $\Lambda' \subseteq \Lambda$ such that $E( \Lambda')$ is still a frame for $L^2 (S)$, see, e.g., \cite[Lemma 3.11]{seip1995on} for a proof that remains valid for general bounded sets. Since $\Lambda' \subseteq \Lambda$, we have $D^-(\Lambda') \leq |S|$, whereas $D^-(\Lambda') \geq |S|$ by the necessary density condition for the frame $E( \Lambda')$. Thus, \Cref{thm:main_intro} is applicable to $E(\Lambda')$.
\\~\\
(ii) The argument is similar to (i). Assuming there would be a discrete set $\Lambda \subseteq \mathbb{R}$ of density $D^+(\Lambda) = |S|$ such that $E(\Lambda)$ is a Riesz sequence, it follows that $\Lambda$ must be separated, and hence part (ii) of \Cref{thm:main_intro} would yield a Riesz basis $E(\Gamma)$ for $L^2 (S)$.
\end{proof}

It was recently shown by Bownik and the second named author that for every set $S \subseteq \mathbb{R}$ of finite measure and for every $\varepsilon > 0$, there exists  a set $\Lambda \subseteq \mathbb{R}$ of density $D^-(\Lambda) \leq (1+\varepsilon) |S|$ such that the system $E(\Lambda)$ is a frame for $L^2 (S)$, cf.\ \cite{bownik2025frame}. \Cref{thm:noframe_critical} shows that this result is optimal: It cannot be improved to a statement on frames at the critical density $D^-(\Lambda) = |S|$.

It is also worth mentioning the situation changes drastically if the frame property is replaced with the weaker property of completeness: As shown by Olevskii and Ulanovskii \cite[Theorem 1]{olevskii2011uniqueness}, every $S \subseteq \mathbb{R}$ of finite measure admits a set $\Lambda \subseteq \mathbb{R}$ of uniform density $D(\Lambda) = |S|$ such that the system $E(\Lambda)$ is complete in $L^2 (S)$.

\subsection{Frames and Riesz sequences over repetitive sets}

As our second application of \Cref{thm:main_intro}, we show that frames and Riesz sequences over repetitive (and more generally almost repetitive) point sets at the critical density are in fact already Riesz bases. Repetitivity is a regularity property of point sets studied in aperiodic order \cite{baake-grimm} and is satisfied, e.g., by point sets arising from primitive substitution rules and Meyer's cut-and-project sets \cite{meyer-book}. Almost repetitivity is a generalization of repetitivity that is suited for point sets not of finite local complexity.

In the proof we will employ the following characterization of almost repetitivity in terms of minimality of the associated hull dynamical system. Namely, a separated and relatively dense set $\Lambda \subseteq \R^d$ is almost repetitive if and only if $W(\Gamma) = W(\Lambda)$ for all $\Gamma \in W(\Lambda)$; see, e.g., \cite[Theorem 3.2]{lagarias-repetitive} and \cite[Theorem 3.11]{frettloh-almost-repetitive}.

\begin{theorem} \label{thm:repetitive}
    Let $S \subseteq \mathbb{R}^d$ be a set of finite measure. Let $\Lambda \subseteq \mathbb{R}^d$ be an almost repetitive set.
    Then the following assertions hold:
    \begin{enumerate}[(i)]
        \item If $E(\Lambda)$ is a frame for $L^2 (S)$ and $D^-(\Lambda) = |S|$, then $E(\Lambda)$ is a Riesz basis for $L^2 (S)$.
            \item If $E(\Lambda)$ is a Riesz sequence in $L^2 (S)$ and $D^+(\Lambda) = |S|$, then $E(\Lambda)$ is a Riesz basis for $L^2 (S)$.
    \end{enumerate}
\end{theorem}
\begin{proof}
(i) Suppose $E(\Lambda)$ is a frame for $L^2(S)$ and $D^-(\Lambda) = |S|$. By \Cref{thm:main_intro} (i) there exists $\Gamma \in W(\Lambda)$ such that $E(\Gamma)$ is a Riesz basis for $L^2(S)$. Since $\Lambda$ is almost repetitive, we have that $\Lambda \in W(\Gamma)$. Since frames and Riesz sequences are stable under weak limits of translates (see, e.g., \Cref{cor:stability_weaklimits}), it follows that $E(\Lambda)$ must be a Riesz basis for $L^2(S)$.
\\~\\
(ii) Suppose that $E(\Lambda)$ is a Riesz sequence in $L^2(S)$ and $D^+(\Lambda) = |S|$. Then \Cref{thm:main_intro} (ii) implies that there exists $\Gamma \in W(\Lambda)$ such that $E(\Gamma)$ is a Riesz basis for $L^2(S)$. We have that $\Lambda \in W(\Gamma)$ as $\Lambda$ is almost repetitive. Since Riesz bases are stable under weak limits of translates (see, e.g., \Cref{cor:stability_weaklimits}), it follows that
$E(\Lambda)$ itself is a Riesz basis for $L^2(S)$.
\end{proof}

A class of cut-and-project sets in $\R^d$ can be constructed as follows: Let $p_1$ and $p_2$ denote the projections of $\R^d \times \R$ onto $\R^d$ and $\R$, respectively. Let $\Gamma$ be a lattice in $\R^d \times \R$ such that the restrictions of $p_1$ and $p_2$ to $\Gamma$ are injective, and such that $p_1(\Gamma)$ and $p_2(\Gamma)$ are dense in $\R^d$ and $\R$, respectively. Letting $I \subseteq \R$ be a half-open interval, the corresponding cut-and-project set is given by
\begin{equation}
    \Lambda(\Gamma,I) = \{ p_1(\gamma) : \gamma \in \Gamma, \; p_2(\gamma) \in I \} . \label{eq:cut-and-project}
\end{equation}
Exponential frames and Riesz sequences built from cut-and-project sets have been studied in, e.g., \cite{grepstad2018riesz, matei2009variant, matei2010simple, kozma2011exponential,  grepstad-multi}. In particular, it was shown by Grepstad and Lev in \cite[Theorem 1.1]{grepstad2018riesz} that for $\Lambda = \Lambda(\Gamma,I)$ as in \eqref{eq:cut-and-project} with $|I| \notin p_2(\Gamma)$ and a Riemann measurable set $S \subseteq \R^d$, the system of exponentials $E(\Lambda)$  cannot be a Riesz basis for $L^2(S)$. Since $\Lambda$ has uniform density equal to $|I|/\covol(\Gamma)$, the following strengthening of their result follows immediately from \Cref{thm:repetitive}.

\begin{theorem} \label{thm:cutandproject}
Let $\Lambda = \Lambda(\Gamma, I)$ be a cut-and-project set as in \eqref{eq:cut-and-project} such that $|I| \notin p_2(\Gamma)$. If $S \subseteq \R^d$ is a Riemann measurable set with
\[ \frac{|I|}{\covol(\Gamma)} = |S|, \]
then $E(\Lambda)$ cannot be a frame for $L^2(S)$, nor a Riesz sequence in $L^2(S)$.
\end{theorem}

\subsection{Organization}
The paper is organized as follows. \Cref{sec:prelim} provides background on general notation, point sets and frames and Riesz sequences that is used throughout the main text. The stability of frames and Riesz sequences and the convergence of associated operators is discussed in \Cref{sec:stability}. The proofs of part (i) and part (ii) of \Cref{thm:main_intro} are given in \Cref{sec:frame} and \Cref{sec:riesz}, respectively.

\section{Preliminaries} \label{sec:prelim}

\subsection{General notation} Throughout this article, we denote  the Euclidean norm of a vector $x \in \mathbb{R}^d$ by $|x| := (|x_1|^2 + ... + |x_d|^2)^{1/2}$ and the associated ball of radius $r>0$ and center $x \in \mathbb{R}^d$ by $B_r (x) = \{ y \in \mathbb{R}^d : |x-y|< r \}$. We also often simply write $B_r$ for $B_r(0)$. The Lebesgue measure of a measurable set $U \subseteq \mathbb{R}^d$ is denoted by $|U|$ and we denote integration against this measure by $\int ... \; dx$.
The space of square-integrable functions on $U$ is denoted by $L^2 (U)$.

We write $e_\xi(t) = e^{2\pi i \, \xi \cdot t}$ for $\xi \in \R^d$ and $t \in \mathbb{R}^d$. For a set $S \subseteq \mathbb{R}^d$ of finite measure, the notation $\| \cdot \|$ and $\langle \cdot, \cdot \rangle$ will always be used to denote the norm and inner product on $L^2 (S)$, respectively. Norms and inner products for other spaces will be written with a subscript.
The Fourier transform $\widehat{f}$ of a function $f \in L^2 (S)$ is defined as
$
\widehat{f} (\xi) = \langle f, e_{\xi} \rangle
$.
We also write $\mathcal{F} : L^2 (S) \to L^2 (\mathbb{R}^d)$ for the unitary operator $f \mapsto \widehat{f}$.

\subsection{Point sets}

A set $\Lambda \subseteq \mathbb{R}^d$ is said to be \emph{relatively separated} if
\[
\Rel(\Lambda) := \sup_{x \in \mathbb{R}^d} \# (\Lambda \cap B_1 (x)) < \infty
\]
and is called \emph{separated} if 
\[
\Sep(\Lambda) := \inf_{\substack{\lambda, \lambda' \in \Lambda \\ \lambda \neq \lambda'}} |\lambda - \lambda'| > 0.
\]
Any separated set is relatively separated, and any relatively separated set is a finite union of separated sets. A set is relatively separated if and only if $D^+(\Lambda) < \infty$.

A set $\Lambda \subseteq \mathbb{R}^d$ is \emph{relatively dense} if there exists $r>0$ such that
\[
\mathbb{R}^d = \bigcup_{\lambda \in \Lambda} B_r (\lambda).
\]
Equivalently, a set $\Lambda$ is relatively dense whenever $D^-(\Lambda) > 0$.

\subsection{Weak limits}
A sequence $\{\Lambda_n\}_{n \in \mathbb{N}}$ of sets $\Lambda_n \subseteq \mathbb{R}^d$ is said to \emph{converge weakly} to a set $\Gamma \subseteq \mathbb{R}^d$ if the following two properties hold:
\begin{enumerate}[(1)]
    \item If $\gamma \in \Gamma$, then there exist $\lambda_n \in \Lambda_n$ such that $\lambda_n \to \gamma$ as $n \to \infty$;
    \item If $(n_k)_{k \in \mathbb{N}}$ is a subsequence of $\mathbb{N}$ and $\lambda_{n_k} \in \Lambda_{n_k}$ with $\lambda_{n_k} \to \gamma \in \mathbb{R}^d$ as $k \to \infty$, then $\gamma \in \Gamma$.
\end{enumerate}
This definition of weak convergence is equivalent to the one used in the introduction.

We denote by $W(\Lambda)$ the set of all weak limits of sequences of translates $\{ \Lambda - x_n \}_{n \in \N}$ of $\Lambda$. If $\Lambda$ is a relatively separated set, then every $\Gamma \in W(\Lambda)$ is also relatively separated, with $\mathrm{rel}(\Gamma) \leq \mathrm{rel}(\Lambda)$. Similarly, if $\Lambda$ is separated, then $\Gamma$ is also separated, with $\Sep(\Gamma) \geq \Sep(\Lambda)$. In addition, if $\Lambda$ is separated and $\Gamma \in W(\Lambda)$, then $D^- (\Lambda) \leq D^-(\Gamma)$ and $D^+ (\Lambda) \geq D^+ (\Gamma)$.

The following lemma is well-known. We provide a proof for completeness.

\begin{lemma}\label{lem:weak-convergence-separated}
Let $\Lambda \subseteq \R^d$ be separated and let $\Lambda_n = \Lambda - x_n$ be a sequence of translates that converges to a set $\Gamma \in W(\Lambda)$. Let $r > 0$ be such that $\Gamma \cap \partial B_r(0) = \emptyset$. Then there exists $n_0, k \in \N$ such that for every $n \geq n_0$ there exist $\mu_n^{(j)} \in \Lambda_n$ and $\gamma^{(j)} \in \Gamma$, $1 \leq j \leq k$, with
\begin{align*}
    \Lambda_n \cap B_r(0) = \{ \mu_n^{(1)}, \ldots, \mu_n^{(k)} \}, \quad \Gamma \cap B_r(0) = \{ \gamma^{(1)}, \ldots, \gamma^{(k)} \},
\end{align*}
and $\mu_n^{(j)} \to \gamma^{(j)}$ as $n \to \infty$ for each $1 \leq j \leq k$.
\end{lemma}

\begin{proof}
The limit set $\Gamma$ must also be separated, and the set $\Gamma \cap B_r$ is finite, say, $\Gamma \cap B_r = \{ \gamma^{(1)}, \ldots, \gamma^{(k)} \}$. Since $\Gamma$ is separated, we may find $\varepsilon > 0$ such that $\Gamma \cap B_\varepsilon(\gamma^{(j)}) = \{ \gamma^{(j)} \}$ for each $1 \leq j \leq k$ and, using that $\Gamma \cap \partial B_r(0) = \emptyset$, we have $B_\varepsilon(\gamma^{(j)}) \subseteq B_r(0)$. Since $\Lambda_n \to \Gamma$, we may find $\mu_n^{(j)} \in \Lambda_n$ such that $\mu_n^{(j)} \to \gamma^{(j)}$ as $n \to \infty$. Then there exists $n_0 \in \N$ such that $\mu_n^{(j)} \in B_\varepsilon(\gamma^{(j)}) \subseteq B_r(0)$ for $n \geq n_0$ so $\{ \mu_n^{(1)}, \ldots, \mu_n^{(k)} \} \subseteq \Lambda_n \cap B_r(0)$. Since $\Gamma$ is separated, the sets $\Lambda_n \cap B_\varepsilon (\gamma^{(j)})$ cannot contain more than one element for $n$ sufficiently large, so after possibly enlarging $n_0$, we have that $\{ \mu_n^{(1)}, \ldots, \mu_n^{(k)} \} = \Lambda_n \cap B_r(0)$.
\end{proof}

\subsection{Frames, Riesz sequences and operators} 
Given a discrete set $\Lambda \subseteq \mathbb{R}^d$ and a set $S \subseteq \mathbb{R}^d$ of finite measure, we consider the exponential system $E(\Lambda)$ in $L^2 (S)$.
We denote the associated \emph{frame operator} of the system $E(\Lambda)$ by
\[
S_{ \Lambda} := \sum_{\lambda \in \Lambda} \langle \cdot, e_{\lambda}  \rangle e_{\lambda} .
\]
The system $E(\Lambda)$ is a \emph{Bessel sequence} if and only if $S_{\Lambda}$ is bounded on $L^2 (S)$. If $E(\Lambda)$ is a Bessel sequence, then $\Lambda$ is relatively separated, and its associated frame operator satisfies the \emph{covariance relation}:
\begin{align} \label{eq:covariance}
    S_{\Lambda} (e_{x} f) = e_x S_{\Lambda - x} f, \quad  x \in \mathbb{R}^d, \; f \in L^2 (S).
\end{align}
Similarly, if $P_{\Lambda}$ denotes the orthogonal projection onto $\Span E(\Lambda)$, then
\begin{align} \label{eq:covariance_projection}
P_{\Lambda} (e_x f) = e_x P_{\Lambda - x} f ,  \quad x \in \mathbb{R}^d, \; f \in L^2 (S).
\end{align}

A system $E(\Lambda)$ is a \emph{frame} for $L^2 (S)$ whenever $S_{\Lambda}$ is bounded and invertible on $L^2 (S)$. A frame is called a \emph{Parseval frame} for $L^2 (S)$ whenever $S_{\Lambda}$ is the identity operator.
If $E(\Lambda)$ is a frame, then the system $\{S^{-1}_{\Lambda} e_{\lambda} \}_{\lambda \in \Lambda}$ is also a frame, called the \emph{canonical dual frame} of $E(\Lambda)$. Similarly, the system $\{S^{-1/2}_{\Lambda} e_{\lambda} \}_{\lambda \in \Lambda}$ is a Parseval frame whenever $E(\Lambda)$ is a frame. If $E(\Lambda)$ is a frame for $L^2 (S)$, then it follows from \eqref{eq:covariance} that its inverse frame operator satisfies
\begin{align} \label{eq:covariance_inverse}
S_{\Lambda}^{-1} (e_{x} f) = e_x S^{-1}_{\Lambda - x} f,  \quad x \in \mathbb{R}^d, \; f \in L^2 (S).
\end{align}

The system $E(\Lambda)$ is a \emph{Riesz sequence} in $L^2 (S)$ if there exist $0 < A \leq B < \infty$ such that
\[
A \| c \|^2_{\ell^2(\Lambda)} \leq \bigg\| \sum_{\lambda \in \Lambda} c_{\lambda} e_{\lambda} \bigg\|^2 \leq B \| c \|^2_{\ell^2 (\Lambda)}
\]
for all $c \in \ell^2 (\Lambda)$. Any Riesz sequence is, in particular, a Bessel sequence. A Riesz sequence that is complete is said to be a \emph{Riesz basis}. A Riesz basis is also a frame.

\section{Stability of frames and Riesz sequences and convergence of operators} \label{sec:stability}
This section is devoted to showing that frames and Riesz sequences are stable under weak limits of translates, and that their associated operators converge in the strong operator topology.

In what follows, we denote by $S \subseteq \mathbb{R}^d$ a set of finite measure. We view $L^2(S)$ as a closed subspace of $L^2(\R^d)$ via extending functions by zero. We will repeatedly use a dense subspace of $L^2 (S)$ defined as follows. If $P : L^2 (\mathbb{R}^d) \to L^2 (\mathbb{R}^d)$ denotes the orthogonal projection onto the closed subspace $\mathcal{F} (L^2 (S)) \subseteq L^2 (\mathbb{R}^d)$, then we define 
\[
\mathcal{D} := \mathcal{F}^{-1} P C_c (\mathbb{R}^d). 
\]
Note that $\mathcal{D}$ is dense in $L^2(S)$.

For proving that frames are stable under weak limits of translates, we will use the following simple lemma involving the local maximal function 
\[
Mf (x) = \sup_{y \in B_1 (0)} |f(x - y)|, \quad x \in \mathbb{R}^d,
\]
of a function $f \in L^2 (\mathbb{R}^d)$.

\begin{lemma} \label{lem:amalgam}
    Let $f \in C (\mathbb{R}^d)$ be a function such that $Mf \in L^2 (\mathbb{R}^d)$ and $r > 1$. If $\Lambda \subseteq \mathbb{R}^d$ is a relatively separated set, then
    \[
    \sum_{\lambda \in \Lambda \setminus B_r (0)} |f(\lambda)|^2 \leq \frac{\Rel(\Lambda)
    }{|B_1 (0)|} \int_{\mathbb{R}^d \setminus B_{r-1} (0)} |Mf(x)|^2 \; dx.
    \]
\end{lemma}

\begin{lemma} \label{lem:densesubspace}
  Let $S \subseteq \mathbb{R}^d$ be a set of finite measure.
  If $f \in \mathcal{D}$, then its Fourier transform $\widehat{f}$ satisfies $M \widehat{f} \in L^2 (\mathbb{R}^d)$ .
\end{lemma}
\begin{proof}
The closed subspace $\mathcal{F}(L^2 (S)) \subseteq L^2 (\mathbb{R}^d)$ is invariant under translation, hence the orthogonal projection $P$ from $L^2 (\mathbb{R}^d)$ onto $\mathcal{F}(L^2 (S))$ commutes with translation. Therefore, 
if $f \in \mathcal{D}$, so that $\widehat{f} = P h$ for some $h \in C_c (\mathbb{R}^d)$, a direct calculation gives
\begin{align*}
\widehat{f}  (\xi)  = \langle \widehat{f} , \widehat{\mathds{1}_S} (\cdot - \xi) \rangle_{L^2 (\mathbb{R}^d)}  = \langle h , (P \widehat{\mathds{1}_S}) (\cdot - \xi) \rangle_{L^2 (\mathbb{R}^d)} = \big(h \ast (\widehat{\mathds{1}_S})^* \big)(\xi),
\end{align*}
where $\ast$ denotes convolution and where $(\widehat{\mathds{1}_S})^* (x) := \overline{\widehat{\mathds{1}_S} (-x)}$ for $x \in \mathbb{R}^d$.
Using that $M(h \ast (\widehat{\mathds{1}_S})^*) \leq M h \ast |\widehat{\mathds{1}_S}|^*$, we thus obtain
\[
\| M \widehat{f} \|_{L^2 (\mathbb{R}^d)}  \leq \| M h \|_{L^1 (\mathbb{R}^d)} \| (\widehat{\mathds{1}_S})^{\ast} \|_{L^2 (\mathbb{R}^d)} =  \| M h \|_{L^1 (\mathbb{R}^d)} |S| < \infty,
\]
which settles the claim.
\end{proof}

\begin{proposition} \label{prop:sot}
 Let $S \subseteq \mathbb{R}^d$ be a set of finite measure and $\Lambda \subseteq \mathbb{R}^d$ be separated.
 Let $\{x_n\}_{n \in \mathbb{N}}$ be a sequence in $\mathbb{R}^d$ such that $\Lambda - x_n$ converges weakly to $\Gamma \in W(\Lambda)$. Then the following assertions hold:
\begin{enumerate}
    \item[(i)] If $E(\Lambda)$ is a Bessel sequence in $L^2 (S)$, then the frame operators $S_{\Lambda - x_n}$ converge in the strong operator topology to the frame operator $S_{\Gamma}$.
    \item[(ii)] If $E(\Lambda)$ is a frame for $L^2 (S)$, then the inverse frame operators $S^{-1}_{\Lambda - x_n}$ converge in the strong operator topology to the inverse frame operator $S^{-1}_{\Gamma}$.
    \item[(iii)] If $E(\Lambda)$ is a Riesz sequence in $L^2 (S)$, then the projections $P_{\Lambda - x_n}$ converge in the strong operator topology to the projection $P_{\Gamma}$.
\end{enumerate}
\end{proposition}
\begin{proof}
Throughout the proof, we simply write $\Lambda_n = \Lambda - x_n$ for $n \in \mathbb{N}$.
\\~\\
(i) Let $0<B <\infty$ be the Bessel bound for $E(\Lambda)$. For $\varepsilon > 0$ and $f \in \mathcal{D}$, a combination of \Cref{lem:amalgam} and \Cref{lem:densesubspace}  yields $r = r (\varepsilon) > 0$ such that
\begin{align} \label{eq:hap}
\sum_{\lambda \in \Lambda' \setminus B_r (0)} | \langle f, e_{\lambda} \rangle |^2 = \sum_{\lambda \in \Lambda' \setminus B_r (0)} | \widehat{f } (\lambda) |^2 < \frac{\varepsilon}{4}
\end{align}
for all relatively separated sets $\Lambda' \subseteq \mathbb{R}^d$ with $\Rel(\Lambda') \leq \Rel(\Lambda)$. Enlarging $r$ if necessary, we may assume that $\Gamma \cap \partial B_r(0) = \emptyset$. The continuity of the map $x \mapsto | \langle f, e_x \rangle |^2$ together with \Cref{lem:weak-convergence-separated} implies that for all sufficiently large $n \in \mathbb{N}$, we have
\[
\bigg| \sum_{\gamma \in \Gamma \cap B_r (0)} |\langle f, e_{\gamma}  \rangle |^2 \; - \sum_{\mu \in \Lambda_n \cap B_r (0)} |\langle f, e_{\mu}  \rangle |^2 \bigg| < \frac{\varepsilon}{2}.
\]
In combination with \eqref{eq:hap}, this shows that
\begin{align} \label{eq:eps_diff}
\bigg| \sum_{\gamma \in \Gamma} |\langle f, e_{\gamma}  \rangle |^2 - \sum_{\mu \in \Lambda_n} |\langle f, e_{\mu}  \rangle |^2 \bigg| \leq \varepsilon
\end{align}
for all sufficiently large $n \in \mathbb{N}$. Therefore, we obtain that $\langle S_{\Lambda_n} f, f \rangle \to \langle S_\Gamma f, f \rangle$ as $n \to \infty$. Using polarization, the density of $\mathcal{D}$ in $L^2(S)$, and the uniform boundedness $ \| S_{\Lambda_n} \| \leq B$, we deduce that $S_{\Lambda_n}$ converges strongly to $S_{\Gamma}$.
\\~\\
(ii) Let $0 < A \leq B < \infty$ be frame bounds of $E(\Lambda)$. Clearly, each system $E(\Lambda_n)$ is also a frame for $L^2 (S)$ with the same bounds, so that $\| S_{\Lambda_n}^{-1} \| \leq 1/A$ for all $n \in \mathbb{N}$. By assertion (i), we have that $S_{\Lambda_n} \to S_{\Gamma}$ in the strong operator topology. Moreover, the identity \eqref{eq:eps_diff} yields that $E(\Gamma)$ is also a frame for $L^2 (S)$, so that $S_{\Gamma}^{-1}$ is well-defined and bounded on $L^2 (S)$.
Combining these observations with the identity
\[
S_{\Lambda_n}^{-1} - S_{\Gamma}^{-1} = S_{\Lambda_n}^{-1} (S_{\Gamma} - S_{\Lambda_n} ) S_{\Gamma}^{-1} ,
\]
it follows therefore for $f \in L^2 (S)$ that
\begin{align*}
\lim_{n \to \infty} \| (S_{\Lambda_n}^{-1} - S_{\Gamma}^{-1}) f \| 
&= \lim_{n \to \infty} \| S_{\Lambda_n}^{-1} (S_{\Gamma} - S_{\Lambda_n} ) S_{\Gamma}^{-1} f \| \\
&\leq A^{-1} \lim_{n \to \infty} \| (S_{\Gamma} - S_{\Lambda_n} ) S_{\Gamma}^{-1} f \| \\
&= 0,
\end{align*}
which finishes the proof of assertion (ii). 
\\~\\
(iii) Let $0 < A \leq B < \infty$ be the Riesz bounds of $E(\Lambda)$. We prove the assertion separately for $f \in \Span E(\Gamma)$ and $f \in (\Span E(\Gamma))^{\perp}$.

First, let $f \in \Span E(\Gamma)$. We may assume that $f = \sum_{\gamma \in \Gamma} c_{\gamma} e_\gamma$ for a finitely supported sequence $c$, say, $\mathrm{supp}(c) \subseteq B_r$. \Cref{lem:weak-convergence-separated}  implies the existence of $n_0 \in \N$ such that $\Lambda_n \cap B_r(0) = \{ \mu_n^{(1)}, \ldots , \mu_n^{(k)} \}$ and $\Gamma \cap B_r(0) = \{ \gamma^{(1)}, \ldots, \gamma^{(k)} \}$ for $n \geq n_0$, with $\mu_n^{(j)} \to \gamma^{(j)}$ as $n \to \infty$ for $1 \leq j \leq k$. For each $n \in \N$, define a $\Lambda_n$-indexed sequence $c^{(n)} \in \ell^2 (\Lambda_n)$ via $c^{(n)}_{\mu_n^{(j)}} = c_{\gamma^{(j)}}$ for $n \geq n_0$ and $c^{(n)}_{\mu} = 0$ otherwise. Set
\[ f_n = \sum_{\mu \in \Lambda_n} c^{(n)}_{\mu} e_\mu . \]
Then
\[ \| P_{\Lambda_n} f - f \| \leq \| f_n - f \| \leq \sum_{j=1}^k | c_{\gamma^{(j)}} | \| e_{\mu_n^{(j)}} - e_{\gamma^{(j)}} \| . \]
Thus, it follows from strong continuity of modulation that $P_{\Lambda_n} f \to f$ as $n \to \infty$.

If $f \in (\Span (E(\Gamma))^{\perp}$, then $P_{\Gamma} f = 0$
and thus it suffices to show that $P_{\Lambda_n} f \to 0$.
Using that each $E(\Lambda_n)$ is a Riesz sequence in $L^2 (S)$ with bounds $0<A\leq B < \infty$, and hence a frame for $\Span E(\Lambda_n)$ with uniform lower frame bound $A>0$, 
it follows that
\begin{align*}
    \| P_{\Lambda_n} f \|^2 \leq \frac{1}{A} \sum_{\mu \in \Lambda_n} |\langle f, e_{\mu} \rangle|^2 .
\end{align*}
Assertion (i), together with the assumption $f \in (\Span (E(\Gamma))^{\perp}$, implies that
\[
\lim_{n \to \infty} \sum_{\mu \in \Lambda_n} |\langle f, e_{\mu} \rangle|^2  = \sum_{\gamma \in \Gamma} |\langle f, e_{\gamma} \rangle |^2 = 0,
\]
and thus $\| P_{\Lambda_n}f \| \to 0$ as $n \to \infty$. This completes the proof.
\end{proof}

\begin{corollary} \label{cor:stability_weaklimits}
Let $S \subseteq \mathbb{R}^d$ be a set of finite measure and $\Lambda \subseteq \mathbb{R}^d$ be separated. Then the following assertions hold:
\begin{enumerate}[(i)]
    \item If $E(\Lambda)$ is a frame for $L^2(S)$ and $\Gamma \in W(\Lambda)$, then $E(\Gamma)$ is a frame for $L^2(S)$ with the same frame bounds.
    \item If $E(\Lambda)$ is a Riesz sequence in $L^2(S)$ and $\Gamma \in W(\Lambda)$, then $E(\Gamma)$ is a Riesz sequence in $L^2(S)$ with the same Riesz bounds.
\end{enumerate}
\end{corollary}

\begin{proof} Let $\{x_n\}_{n \in \mathbb{N}}$ be sequence in $\mathbb{R}^d$ such that $\Lambda - x_n \to \Gamma$. Set $\Lambda_n = \Lambda - x_n$. 
\\~\\
(i) If $E(\Lambda)$ is a frame for $L^2 (S)$ with bounds $0<A \leq B < \infty$, then so is any $E(\Lambda_n)$ with $n \in \mathbb{N}$, hence $AI \leq S_{\Lambda_n} \leq BI$ for each $n \in \N$. Since $S_{\Lambda_n}$ converges to $S_{\Gamma}$ in the strong operator topology by \Cref{prop:sot} (i), it follows that $AI \leq S_\Gamma \leq BI$, which proves the claim.
\\~\\
(ii) Let $c = \{ c_{\gamma} \}_{\gamma \in \Gamma}$ be a finitely supported sequence, say $\mathrm{supp}(c) = \{ \gamma^{(1)}, \ldots, \gamma^{(k)} \}$, and choose elements $\mu_n^{(j)} \in \Lambda_n$ such that $\mu_n^{(j)} \to \gamma^{(j)}$ as $n \to \infty$. Then $e_{\mu_n^{(j)}} \to e_{\gamma^{(j)}}$ as $n \to \infty$ by strong continuity of modulation. Since also $P_{\Lambda_n} e_{\gamma^{(j)}} \to e_{\gamma^{(j)}}$ by \Cref{prop:sot} (iii), we infer that
\[ \Big\| \sum_{\gamma \in \Gamma}c_{\gamma} e_\gamma \Big\|^2 = \lim_{n \to \infty} \Big\| P_{\Lambda_n} \sum_{j=1}^k c_{\gamma^{(j)}} e_{\gamma^{(j)}} \Big\|^2 = \lim_{n \to \infty} \Big\|  \sum_{j=1}^k c_{\gamma^{(j)}} e_{\mu_n^{(j)}} \Big\|^2 . \]
We may now apply the Riesz bounds of $E(\Lambda_n)$ to the right hand side to reach the desired conclusion.
\end{proof}

\section{Frames at the critical density} \label{sec:frame}

In this section, we will prove part (i) of \Cref{thm:main_intro}. For this, we will use the notion of the \emph{upper frame measure} of a frame $E(\Lambda)$ for $L^2 (S)$ defined by 
\[ M^+(\Lambda) := \limsup_{r \to \infty} \sup_{x \in \mathbb{R}^d} \frac{1}{\# (\Lambda \cap B_r (x))} \sum_{\lambda \in \Lambda \cap B_r (x)} \langle e_{\lambda}, S_{\Lambda}^{-1} e_{\lambda} \rangle , \]
where $S_\Lambda$ denotes the frame operator of $E(\Lambda)$.

The following theorem relates the upper frame measure and lower Beurling density. It  can be found implicitly for bounded spectra in (the proof of) \cite[Theorem 4]{nitzan2012revisiting} (see also 
\cite{kolountzakis1996structure, iosevich2006weyl}). The general case is proven in \cite{bownik2025frame}; cf. \cite[Theorem 6.1]{bownik2025frame}.

\begin{theorem}[\cite{nitzan2012revisiting, bownik2025frame}] \label{thm:framemeasure}
    Let $S \subseteq \mathbb{R}^d$ be a set of finite measure and let $\Lambda \subseteq \mathbb{R}^d$. If $E(\Lambda)$ is a frame for $L^2 (S)$, then
    \[
    M^+(\Lambda) = \frac{|S|}{D^- (\Lambda)}. 
    \]
\end{theorem}

Using \Cref{thm:framemeasure}, we will prove the following theorem.

\begin{theorem} \label{thm:frame-to-riesz-basis}
    Let $S \subseteq \mathbb{R}^d$ be a set of finite measure and $\Lambda \subseteq \mathbb{R}^d$ be separated.  Then the following are equivalent:
    \begin{enumerate}[(i)]
        \item There exists $\Gamma \in W(\Lambda)$ such that $E(\Gamma)$ is a Riesz basis for $L^2(S)$.
        \item $D^-(\Lambda) = |S|$.
        \item For every $\varepsilon > 0$ and every $r>0$ there exists $x \in \mathbb{R}^d$ such that
        \[
        \langle e_{\lambda}, S_{\Lambda}^{-1} e_{\lambda} \rangle > 1 - \varepsilon  \quad \text{for all} \quad \lambda \in \Lambda \cap B_r (x). \]
    \end{enumerate}
\end{theorem}

\begin{proof}
Recall that if $\Lambda \subseteq \mathbb{R}^d$ is such that $E(\Lambda)$ is a frame for $L^2 (S)$, then $D^- (\Lambda) \geq |S|$.
\\~\\
(i) $\Rightarrow$ (ii).  If there exists $\Gamma \in W(\Lambda)$  such that $E(\Gamma)$ is a Riesz basis for $L^2 (S)$, then necessarily $D^-(\Gamma) = |S|$.  Since $D^-(\Gamma) \geq D^-(\Lambda)$, we get that $|S| \geq D^-(\Lambda)$, and thus $D^-(\Lambda) = |S|$, which shows (ii).
\\~\\
(ii) $\Rightarrow$ (iii). 
Suppose that  $D^-(\Lambda) = |S|$. Given $\varepsilon > 0 $, we consider the set
    \[ \Lambda_{\varepsilon} := \big\{ \lambda \in \Lambda : \langle e_{\lambda} , S^{-1}_{\Lambda} e_{\lambda} \rangle \leq 1 - \varepsilon \big\}. \]
    For showing (iii), we need to show that $\Lambda_{\varepsilon}$ is not relatively dense. Indeed, this means precisely that for every $r > 0$ there exists $x \in \R^d$ such that $\Lambda_\varepsilon \cap B_r(x) = \emptyset$, that is,\ $\langle e_\lambda, S_\Lambda^{-1} e_\lambda \rangle > 1 - \varepsilon$ for all $\lambda \in \Lambda \cap B_r(x)$.
    
    Since $D^-(\Lambda) = |S|$ by assumption, an application of \Cref{thm:framemeasure} implies that $M^+ (\Lambda) =1$,
    which is equivalent to
    \begin{align} \label{eq:framemeasure}
      \liminf_{r \to \infty} \inf_{x \in \mathbb{R}^d} \frac{1}{\# (\Lambda \cap B_r (x))} \sum_{\lambda \in \Lambda \cap B_r (x)} (1-\langle e_{\lambda}, S_{\Lambda}^{-1} e_{\lambda} \rangle) =  0.
    \end{align}
    For $r > 0$ and $x \in \mathbb{R}^d$, a direct calculation gives
    \begin{align*}
    \frac{1}{\#(\Lambda \cap B_r (x))} 
    \sum_{\lambda \in \Lambda \cap B_r(x)} (1 - \langle e_{\lambda} , S_{\Lambda}^{-1} e_{\lambda} \rangle ) 
    &\geq \frac{1}{\#(\Lambda \cap B_r (x))} 
    \sum_{\lambda \in \Lambda_{\varepsilon} \cap B_r(x)} ( 1 - \langle e_{\lambda} , S_{\Lambda}^{-1} e_{\lambda} \rangle ) \\
    &\geq \varepsilon \frac{\#(\Lambda_{\varepsilon} \cap B_r (x))}{\# (\Lambda \cap B_r (x))}.
    \end{align*}
    Using \eqref{eq:framemeasure}, we thus obtain
  \begin{align*}
    0 = \liminf_{r \to \infty} \inf_{x \in \mathbb{R}^d} \frac{\#(\Lambda_{\varepsilon} \cap B_r (x))}{\# (\Lambda \cap B_r (x))} &=  \liminf_{r \to \infty} \inf_{x \in \mathbb{R}^d} \frac{\#(\Lambda_{\varepsilon} \cap B_r (x))}{|B_r (x)|} \frac{|B_r (x)|}{\# (\Lambda \cap B_r (x))} \\
    &\geq \frac{D^- (\Lambda_{\varepsilon})}{D^+ (\Lambda)}.
    \end{align*}
    Since $0 < D^+ (\Lambda) < \infty$ as $E(\Lambda)$ is a frame for $L^2 (S)$, it follows that $D^-(\Lambda_{\varepsilon}) = 0$. From this we conclude that $\Lambda_\varepsilon$ is not relatively dense.
\\~\\
(iii) $\Rightarrow$ (i).
    Assume that (iii) holds. Then, for each $n \in \N$, we can choose $x_n \in \mathbb{R}^d$ such that
    \[
    1 - \frac{1}{n} < \langle e_{\lambda}, S_{\Lambda}^{-1} e_{\lambda} \rangle \quad \text{for all} \quad \lambda \in \Lambda \cap B_n (x_n).
    \]
    By passing to a subsequence if necessary, we may assume that $\Lambda - x_n$ converges to a set $\Gamma \in W(\Lambda)$ as $n \to \infty$. Let $\gamma \in \Gamma$. Then there exist $\lambda_n  \in \Lambda$ such that $\lambda_n - x_n \to \gamma$ as $n \to \infty$. Thus, given $\delta > 0$, there exists $N_1 \in \mathbb{N}$ such that $\lambda_n - x_n \in B_{\delta} (\gamma)$ for all $n \geq N_1$. Choose $N_2 \in \mathbb{N}$ such that $B_{\delta} (\gamma) \subseteq B_{N_2}(0)$. Then, for $n \geq \max\{ N_1, N_2\}$, we have that $\lambda_n - x_n \in B_n (0)$, 
    and hence 
    $
     \lambda_n  \in \Lambda \cap B_n (x_n), 
    $
    which implies that 
    \[
    1 - \frac{1}{n} < \langle e_{\lambda_n }, S^{-1}_{\Lambda} e_{\lambda_n } \rangle \quad \text{for all}  \quad n \geq \max\{ N_1, N_2 \}.
    \]
    Since 
    $
     S_{\Lambda - x_n}^{-1} \to S_{\Gamma}^{-1}
    $ as $n \to \infty$ in the strong operator topology by \Cref{prop:sot} (ii) and  $e_{\lambda_n - x_n} \to e_{\gamma}$ as $n \to \infty$ by strong continuity of modulation, it follows from the above and the identity \eqref{eq:covariance_inverse} that
    \[
    \langle e_{\gamma} , S_{\Gamma}^{-1} e_{\gamma} \rangle = \lim_{n \to \infty} \langle e_{\lambda_n - x_n}, S_{\Lambda - x_n}^{-1} e_{\lambda_n - x_n} \rangle = \lim_{n \to \infty} \langle e_{\lambda_n}, S^{-1}_{\Lambda} e_{\lambda_n} \rangle \geq 1.
    \]
    Since $\gamma \in \Gamma$ was arbitrary, we deduce that $ \langle e_{\gamma} , S_{\Gamma}^{-1} e_{\gamma} \rangle = 1$ for all $\gamma \in \Gamma$. 
    Hence, the canonical Parseval frame $S_{\Gamma}^{-1/2} E(\Gamma)$ of $E(\Gamma)$ has unit norm, and thus must be an orthonormal basis. In turn, this implies
    that $E(\Gamma)$ is a Riesz basis for $L^2(S)$, which completes the proof.
\end{proof}

\section{Riesz sequences at the critical density} \label{sec:riesz}

In this section we prove part (ii) of \Cref{thm:main_intro}. Recall that for a set $\Lambda \subseteq \R^d$, we denote by $P_\Lambda$ the orthogonal projection onto $\Span E(\Lambda)$.

We shall use the following theorem, which can be understood as an analogue of  \Cref{thm:framemeasure} for Riesz sequences.
Its proof is modeled on the necessary density condition for uniformly minimal systems given in \cite[Theorem 2]{nitzan2012revisiting}. 
Recall that a system $E(\Lambda)$ is said to be \emph{uniformly minimal} if it has a biorthogonal system whose norms are uniformly bounded. 

\begin{theorem}\label{prop:riesz-sequence-density}
Let $S \subseteq \mathbb{R}^d$ be a set of finite measure. 
Suppose that $E(\Lambda)$ is uniformly minimal in $L^2(S)$. Then
\begin{align*}
    D^+(\Lambda) &= \limsup_{r \to \infty} \sup_{x \in \R^d} \frac{1}{|B_r(x)|} \int_{B_r(x)} \| P_{\Lambda \cap B_r (x)} e_y \|^2 \dif{y}.
\end{align*}
\end{theorem}
\begin{proof}
Let $\varepsilon > 0$ and choose $r_0 = r_0(\varepsilon) > 0$ such that 
\begin{align} \label{eq:continuous_tail}
\int_{\mathbb{R}^d \setminus B_{r_0}} |\widehat{\mathds{1}_S} (y)|^2 \; dy \leq \varepsilon^2. 
\end{align} 
Let $\{ h_{\lambda} \}_{\lambda \in \Lambda}$ be a  biorthogonal system for $E(\Lambda)$ satisfying $C := \sup_{\lambda \in \Lambda} \| h_{\lambda} \| < \infty$. 
Fix $r > 0$. Then the systems $E(\Lambda \cap B_r (x))$ and $\{P_{\Lambda \cap B_r (x)} h_{\lambda} \}_{\lambda \in \Lambda \cap B_r (x)} $ are biorthogonal in $\Span E(\Lambda \cap B_r (x))$, because
\[
\langle e_{\lambda}, P_{\Lambda \cap B_r (x)} h_{\lambda'} \rangle = \langle e_{\lambda}, h_{\lambda'} \rangle = \delta_{\lambda, \lambda'}, \quad \lambda, \lambda' \in \Lambda \cap B_r (x).
\]
Using the biorthogonality of $E(\Lambda \cap B_r (x))$ and $\{P_{\Lambda \cap B_r (x)} h_{\lambda} \}_{\lambda \in \Lambda \cap B_r (x)} $, we have that
\begin{align*}
\# (\Lambda \cap B_r (x))  &= \sum_{\lambda \in \Lambda \cap B_r (x)} \langle \mathcal{F} e_{\lambda} , \mathcal{F} P_{\Lambda \cap B_r (x)} h_{\lambda} \rangle_{L^2 (\mathbb{R}^d)} \\
&= \sum_{\lambda \in \Lambda \cap B_r (x)} \int_{\mathbb{R}^d}  (\mathcal{F} e_{\lambda}) (y) \overline{ (\mathcal{F} P_{\Lambda \cap B_r (x)} h_{\lambda})(y)}  \; dy.
\end{align*}
Moreover, the biorthogonality of $E(\Lambda \cap B_r (x))$ and $\{P_{\Lambda \cap B_r (x)} h_{\lambda} \}_{\lambda \in \Lambda \cap B_r (x)} $ also yields
\begin{align*}
\sum_{\lambda \in \Lambda \cap B_r (x)} (\mathcal{F} e_{\lambda}) (y) \overline{ (\mathcal{F} P_{\Lambda \cap B_r (x)} h_{\lambda})(y)}
&= \sum_{\lambda \in \Lambda \cap B_r (x)}   \langle e_y, P_{\Lambda \cap B_r (x)} h_{\lambda} \rangle \langle e_{\lambda}, e_{y} \rangle \\
&= \| P_{\Lambda \cap B_r (x)} e_y \|^2 
\end{align*}
for any $y \in \mathbb{R}^d$. Therefore, we have that
\begin{align*}
\# (\Lambda \cap B_r (x)) &= \int_{B_r (x)} \| P_{\Lambda \cap B_r (x)} e_y \|^2 \; dy + \int_{B_{r+r_0}(x) \setminus B_r (x)} \| P_{\Lambda \cap B_r (x)} e_y \|^2 \; dy \\
&\quad \quad +  \sum_{\lambda \in \Lambda \cap B_r (x)} \int_{\mathbb{R}^d \setminus B_{r+r_0} (x)}  (\mathcal{F} e_{\lambda}) (y) \overline{ (\mathcal{F} P_{\Lambda \cap B_r (x)} h_{\lambda})(y)} dy ,
\end{align*}
which yields that
\begin{align*}
&\bigg| \# (\Lambda \cap B_r (x)) - \int_{B_r(x)} \| P_{\Lambda \cap B_r (x)} e_y \|^2 \; dy \bigg| \numberthis \label{eq:estimate_difference} \\
&\quad \quad \leq \int_{B_{r+r_0}(x) \setminus B_r (x)} \| P_{\Lambda \cap B_r (x)} e_y \|^2 \; dy\\
&\quad \quad \quad \quad + \sum_{\lambda \in \Lambda \cap B_r (x)} \bigg| \int_{\mathbb{R}^d \setminus B_{r+r_0} (x)}  (\mathcal{F} e_{\lambda}) (y) \overline{ (\mathcal{F} P_{\Lambda \cap B_r (x)} h_{\lambda})(y)} dy \bigg|
\end{align*}
For further estimating \eqref{eq:estimate_difference}, first simply note that 
\begin{align} \label{eq:estimate_firstterm}
\int_{B_{r+r_0}(x) \setminus B_r (x)} \| P_{\Lambda \cap B_r (x)} e_y \|^2 \; dy \leq |B_{r+r_0}(x) \setminus B_r(x)| |S|. 
\end{align}
Second, note that for fixed $\lambda \in \Lambda \cap B_r (x)$, we have that $B_{r_0}(\lambda) \subseteq B_{r+r_0}(x)$. Hence, it follows from \eqref{eq:continuous_tail} that
\begin{align*}
&\bigg| \int_{\mathbb{R}^d \setminus B_{r+r_0} (x)}  (\mathcal{F} e_{\lambda}) (y) \overline{( \mathcal{F} P_{\Lambda \cap B_r (x)} h_{\lambda})(y)} dy \bigg| \\
&\quad \quad \leq \bigg( \int_{\mathbb{R}^d \setminus B_{r+r_0} (x)} | (\mathcal{F} e_{\lambda}) (y) |^2 \; dy \bigg)^{1/2} \| \mathcal{F} P_{\Lambda \cap B_r (x)} h_{\lambda} \|_{L^2 (\mathbb{R}^d)}  \\
&\quad \quad \leq \bigg( \int_{\mathbb{R}^d \setminus B_{r_0} (\lambda)} | \widehat{\mathds{1}_S} (y - \lambda) |^2 \; dy \bigg)^{1/2} \|  h_{\lambda} \| \\
&\quad \quad \leq \varepsilon C.
\end{align*}
Therefore, 
\begin{align} \label{eq:sum_tail_estimate}
 \sum_{\lambda \in \Lambda \cap B_r (x)} \bigg| \int_{\mathbb{R}^d \setminus B_{r+r_0} (x)}  (\mathcal{F} e_{\lambda}) (y) \overline{( \mathcal{F} P_{\Lambda \cap B_r (x)} h_{\lambda})(y)} dy \bigg| \leq \varepsilon C \# (\Lambda \cap B_r (x)). 
\end{align}

Combining the estimates \eqref{eq:estimate_firstterm} and \eqref{eq:sum_tail_estimate}, together with the estimate of \eqref{eq:estimate_difference}, we obtain
\begin{align*}
&\bigg| \frac{\#(\Lambda \cap B_r (x))}{|B_r(x)|} - \frac{1}{|B_r(x)|} \int_{B_r(x)} \| P_{\Lambda \cap B_r (x)} e_y \|^2 \; dy \bigg| \\
&\quad \quad \leq \frac{|B_{r+r_0} (x) \setminus B_{r} (x)|}{|B_r(x)|} |S| + \varepsilon C \frac{\# (\Lambda \cap B_r (x))}{|B_r (x)|}.
\end{align*}
Hence, taking the supremum over $x \in \mathbb{R}^d$ and then taking the limit supremum as $r \to \infty$, we obtain that
\[ \limsup_{r \to \infty} \sup_{x \in \mathbb{R}^d} \bigg| \frac{\#(\Lambda \cap B_r (x))}{|B_r(x)|} - \frac{1}{|B_r(x)|} \int_{B_r(x)} \| P_{\Lambda \cap B_r (x)} e_y \|^2 \; dy \bigg| \leq \varepsilon D^+ (\Lambda) .\]
Since $D^+ (\Lambda) < \infty$ and since $\varepsilon > 0$ was chosen arbitrary, this yields that
\[ \limsup_{r \to \infty} \sup_{x \in \mathbb{R}^d} \bigg| \frac{\#(\Lambda \cap B_r (x))}{|B_r(x)|} - \frac{1}{|B_r(x)|} \int_{B_r(x)} \| P_{\Lambda \cap B_r (x)} e_y \|^2 \; dy \bigg| = 0 ,\]
which easily yields the claim.
\end{proof}

The following lemma provides a continuous analogue of the fact that discrete sets $\Lambda \subseteq \R^d$ are not relatively dense if and only if $D^-(\Lambda) = 0$.

\begin{lemma}\label{lem:continuous-relatively-dense}
Let $Q \subseteq \R^d$ be measurable. Then the following are equivalent:
\begin{enumerate}[(i)]
    \item \[ \liminf_{r \to \infty} \inf_{x \in \R^d} \frac{| Q \cap B_r(x)|}{|B_r|} = 0. \]
    \item For every $r > 0$ and $\delta > 0$, there exists $x \in \R^d$ such that
    \[ | Q \cap B_r(x)| < \delta . \]
\end{enumerate}
\end{lemma}

\begin{proof}
 We prove that (i) implies (ii) by contraposition. Hence, assume there exists $r_0, \delta_0 > 0$ such that
\begin{align} \label{eq:contrapositive}
| Q \cap B_{r_0}(x) | \geq \delta_0 \quad \text{ for all } x \in \R^d.  \end{align}
For $r > 0$ and $x, y, z \in \mathbb{R}^d$, we have the following inequality of indicator functions:
\begin{align} \label{eq:claim} \mathds{1}_{B_{r_0} (x)}(z) \mathds{1}_{B_r(y)}(x) \leq \mathds{1}_{B_{r+r_0}(y)}(z) \mathds{1}_{B_{r_0}(z)}(x).
\end{align}
For fixed $r > 0$ and $y \in \mathbb{R}^d$, integrating the inequality \eqref{eq:contrapositive} over $B_r (y)$ and using \eqref{eq:claim} gives
\begin{align*} 
    \delta_0 |B_r| 
    &\leq \int_{B_r(y)} |Q \cap B_{r_0}(x)| \dif{x} \\
    &= \int_{\R^d} \int_{Q} \mathds{1}_{B_{r_0}(x)}(z) \mathds{1}_{B_r(y)}(x) \dif{z} \dif{x} \\
    &\leq \int_{Q} \int_{\R^d} \mathds{1}_{B_{r+r_0}(y)}(z) \mathds{1}_{B_{r_0}(z)}(x) \dif{x} \dif{z} \\
    &= |B_{r_0}| |Q \cap B_{r+r_0}(y)| .
\end{align*}
Hence, for $r>0$ and $y \in \mathbb{R}^d$,
\[ \frac{|Q \cap B_{r+r_0}(y)|}{|B_r|} \geq \frac{\delta_0}{|B_{r_0}|}  \]
Since $|B_{r+r_0}|/|B_r| \to 1$ as $r \to \infty$, taking the infimum over $y \in \mathbb{R}^d$ and taking the limit infimum as $r \to \infty$, it follows that (i) cannot hold.

The fact that (ii) implies (i) is immediate.
\end{proof}

Using the previous results, we now prove the following theorem. 

\begin{theorem}\label{thm:riesz-sequence-to-riesz-basis}
Let $S \subseteq \mathbb{R}^d$ be a set of finite measure and let $E(\Lambda)$ be a Riesz sequence in $L^2(S)$. Then the following assertions are equivalent:
\begin{enumerate}[(i)]
    \item There exists $\Gamma \in W(\Lambda)$ such that $E(\Gamma)$ is a Riesz basis for $L^2(S)$.
    \item $D^+(\Lambda) = |S|$.
    \item For every $\varepsilon > 0$, $\delta > 0$ and $r > 0$, there exists $x \in \R^d$ such that
    \[ | \{ y \in B_r(x) : \| (I - P_\Lambda) e_y \| \geq \varepsilon \} | < \delta . \]
\end{enumerate}
\end{theorem}

\begin{proof}
Let $\Lambda \subseteq \mathbb{R}^d$ be such that $E(\Lambda)$ is a Riesz sequence. Recall that $\Lambda$ must be separated and $D^+ (\Lambda) \leq |S|$. Throughout the proof, we use the notation
\[ Q_\varepsilon := \big\{ y \in \R^d : \| (I - P_\Lambda) e_y \| \geq \varepsilon \big\}   \]
for $\varepsilon > 0$.
\\~\\
(i) $\Rightarrow$ (ii).
If $\Gamma \in W(\Lambda)$ is such that $E(\Gamma)$ is a Riesz basis for $L^2 (S)$, then $D^+ (\Gamma) = |S|$. In addition, we have that $D^+(\Gamma) \leq D^+(\Lambda)$. Hence, we conclude that $D^+(\Lambda) = |S|$.
\\~\\
(ii) $\Rightarrow$ (iii). 
Suppose that  $D^+(\Lambda) = |S|$. Using \Cref{prop:riesz-sequence-density} and the inequality $\| P_{\Lambda \cap B_r(x)} e_y \|^2 \leq \| P_\Lambda e_y \|^2$ for $x,y \in \mathbb{R}^d$ and $r>0$, we obtain that
\[ \liminf_{r \to \infty} \inf_{x \in \R^d} \frac{1}{|B_r(x)|} \int_{B_r(x)} \| (I- P_{\Lambda}) e_y \|^2 \dif{y} = 0 . \]
In particular, for every $\varepsilon > 0$, we have that
\[ \liminf_{r \to \infty} \inf_{x \in \R^d} \frac{|Q_\varepsilon \cap B_r(x)| }{|B_r(x)|} \geq 0.  \]
By \Cref{lem:continuous-relatively-dense}, this implies (iii).
\\~\\
(iii) $\Rightarrow$ (i). For each $n \in \mathbb{N}$, the strong continuity of modulation implies the existence of $0 < \delta_n < 1$ such that $\| e_z - e_w \| < 1/n$ whenever $|z-w| < \delta_n$. By the assumption of assertion (iii), we may find $x_n \in \R^d$ such that
\begin{equation}
     |Q_{1/n} \cap B_n(x_n) | < |B_{\delta_n}(0)| . \label{eq:inequality-contradict}
\end{equation}
Passing to a subsequence if necessary, we may assume that $\{ \Lambda  - x_n \}_{n \in \N}$ converges weakly to a set $\Gamma \in W(\Lambda)$. Then $P_{ \Lambda - x_n}$ converges in the strong operator topology to $P_{\Gamma}$ by \Cref{prop:sot} (iii). 

Let $x \in \mathbb{R}^d$. Choose $n_0 \in \N$ such that $x \in B_{n-1}(0)$ for all $n \geq n_0$, and set $y_n := x + x_n \in B_{n-1}(x_n)$ for $n \geq n_0$.
We claim that $\mathrm{dist}(y_n, (Q_{1/n})^c) \leq \delta_n$ for all $n \geq n_0$. If not, then $\mathrm{dist}(y_n, (Q_{1/n})^c) > \delta_n$ implies that $B_{\delta_n}(y_n) \subseteq Q_{1/n}$. 
Since $y_n \in B_{n-1}(x_n)$ for all $n \geq n_0$, we also have that $B_{\delta_n}(y_n) \subseteq B_{\delta_n + n-1}(x_n) \subseteq B_n(x_n)$ for all $n \geq n_0$. Therefore, it follows that $B_{\delta_n}(y_n) \subseteq Q_{1/n} \cap B_{n}(x_n)$, and hence
\[ 
|B_{\delta_n}(0)| \leq |Q_{1/n} \cap B_n(x_n)|, 
\]
which would contradict \eqref{eq:inequality-contradict}. We conclude that $\mathrm{dist}(y_n,(Q_{1/n})^c) \leq \delta_n$ for each $n \geq n_0$. Using this, for $n \geq n_0$, we can find $z_n \in \R^d$ with $\| (I - P_\Lambda) e_{z_n} \| < 1/n$ and such that $|y_n - z_n| \leq \delta_n$. Hence, a direct calculation using the identity \eqref{eq:covariance_projection} entails
\begin{align*}
    \| (I - P_{\Gamma})e_{x} \| &= \lim_{n \to \infty} \| (I - P_{\Lambda - x_n}) e_x \| = \lim_{n \to \infty} \| (I - P_{\Lambda}) e_{y_n} \| \\
    &\leq \lim_{n \to \infty} \big(  \| (I - P_\Lambda) (e_{y_n} - e_{z_n}) \| + \| (I - P_\Lambda) e_{z_n} \| \big) \\
    &\leq \lim_{n \to \infty} \big( \frac{1}{n} + \frac{1}{n} \big) = 0 .
\end{align*}

Since $x \in \mathbb{R}^d$ was chosen arbitrary and $\{ e_x : x \in \R^d \}$ is complete in $L^2(S)$, it follows that $P_\Gamma = I$, hence $E(\Gamma)$ is a Riesz basis for $L^2(S)$.
\end{proof}

\section*{Acknowledgements}
The second author thanks Lukas Liehr for sharing the reference \cite{olevskii2013sampling}.
For J. v. V., this research was funded in whole or in part by the
Austrian Science Fund (FWF): 10.55776/PAT2545623.

\bibliographystyle{abbrv}
\bibliography{bibl}

\end{document}